\documentclass[11pt]{article}
\usepackage{amsthm}
\usepackage{amssymb}
\usepackage{amsmath,enumerate}
\usepackage{comment}
\usepackage{thm-restate}
\usepackage{url} 
\usepackage{hyperref}
\usepackage{graphicx}
\usepackage{subfigure}
\usepackage[noabbrev,capitalise]{cleveref}
\usepackage[affil-it]{authblk}
\usepackage{color}
\usepackage[normalem]{ulem}
\usepackage{tikz}

\usepackage[margin=1in]{geometry}

\renewcommand{\v}{\textup{\textsf{v}}}

\theoremstyle{plain}
\newtheorem{thm}{Theorem}[section]
\newtheorem{lem}[thm]{Lemma}

\newtheorem{rem}[thm]{Remark}
\newtheorem{conj}[thm]{Conjecture}

{\noindent \emph{Proof.} {}{#1}{}}{\hfill
	$\Diamond$\vspace{1em}}

\theoremstyle{plain} 
\newcommand{\thistheoremname}{}
\newtheorem{genericthm}[section]{\thistheoremname}

\theoremstyle{definition}

\def\less{\setminus}

\newcounter{counter}

\def\ket{\mathcal{K}_8^{-3}}

\newcommand{\km}[2]{\mathcal{K}_{#1}^{-{#2}}}
\def\kns{\km{9}{6}}
\def\knf{\km{9}{5}}
\def\se{\succcurlyeq}

\def\dfn#1{{\sl #1}}

\title{Every graph with no $\mathcal{K}_9^{-6}$ minor is $8$-colorable}
\author{Michael Lafferty\thanks{Department  of Mathematics, University of Central Florida, Orlando, FL 32816, USA. Supported in part by NSF grant DMS-2153945.}\hskip 1cm Zi-Xia Song\thanks{Department  of Mathematics, University of Central Florida, Orlando, FL 32816, USA. Supported by  NSF grant    DMS-2153945. Email: Zixia.Song@ucf.edu.}}
 
\date{October 5, 2022}
\begin{document}
\maketitle
 
\begin{abstract}
 
 For positive integers $t$ and $s$, let $\mathcal{K}_t^{-s}$ denote the family of graphs obtained from the complete graph $K_t$ by removing $s$ edges. A graph $G$ has no  $\mathcal{K}_t^{-s}$ minor  if it has no $H$ minor for every $H\in  \mathcal{K}_t^{-s}$. Motivated by  the famous Hadwiger's Conjecture,  Jakobsen in 1971 proved that every   graph  with no   $\mathcal{K}_7^{-2}$   minor is $6$-colorable; very recently the present authors  proved that every graph with no $\mathcal{K}_8^{-4}$ minor is $7$-colorable. In this paper we  continue our work and prove that every graph with no $\mathcal{K}_9^{-6}$ minor  is $8$-colorable. Our result  implies that $H$-Hadwiger's Conjecture, suggested by Paul Seymour in 2017,  is true for all graphs $H$ on nine vertices such that $H$ is  a subgraph of  every graph in $ \mathcal{K}_9^{-6}$.
\end{abstract}

\baselineskip 16pt

\section{Introduction}

All graphs in this paper are finite and undirected, and have no loops or parallel edges.   For a graph $G$ we use $|G|$, $e(G)$, $\delta (G)$, $\Delta(G)$, $\alpha(G)$, $\chi(G)$ to denote the number
of vertices, number of edges,   minimum degree, maximum degree,  independence number, and chromatic number  of $G$, respectively.  The \dfn{complement} of $G$ is denoted by $\overline{G}$.  For any positive integer k, we define   $[k]$ to be the set $\{1, \ldots, k\}$. A graph  $H$ is a \dfn{minor} of a graph $G$ if  $H$ can be
 obtained from a subgraph of $G$ by contracting edges.  We write $G\se H$ if 
$H$ is a minor of $G$.
In those circumstances we also say that  $G$ has an  \dfn{$H$ minor}.  For positive integers  $t, s$, we use $\mathcal{K}_t^{-s}$ to denote  the family of graphs obtained from the complete graph $K_t$ by deleting $s$ edges. We use $K_t^-$, $K_t^=$, and $K_t^\equiv$ to denote the unique graph  obtained from $K_t$ by deleting one, two and three  independent edges, respectively; and 
  $K_t^<$ to denote the unique graph obtained from $K_t$ by deleting two   adjacent edges. Note that $\mathcal{K}_t^{-1}=\{K_t^-\}$ and $\mathcal{K}_t^{-2}=\{K_t^=, K_t^<\}$. 
A graph $G$ has \dfn{no  $\mathcal{K}_t^{-s}$ minor}  if it has no $H$ minor for every $H\in  \mathcal{K}_t^{-s}$; and $G$ has a  $\mathcal{K}_t^{-s}$ minor, otherwise. We write $G\se \mathcal{K}_t^{-s}$ if 
$G$ has a  $\mathcal{K}_t^{-s}$ minor.\medskip

Our work is motivated by   Hadwiger's Conjecture~\cite{Had43}, which is perhaps the most famous conjecture in graph theory.

\begin{conj}[Hadwiger's Conjecture~\cite{Had43}]\label{HC} Every graph with no $K_t$ minor is $(t-1)$-colorable. 
\end{conj}

\cref{HC} is  trivially true for $t\le3$, and reasonably easy for $t=4$, as shown independently by Hadwiger~\cite{Had43} and Dirac~\cite{Dirac52}. However, for $t\ge5$, Hadwiger's Conjecture implies the Four Color Theorem~\cite{AH77,AHK77, RSST97}.   Wagner~\cite{Wagner37} proved that the case $t=5$ of Hadwiger's Conjecture is, in fact, equivalent to the Four Color Theorem, and the same was shown for $t=6$ by Robertson, Seymour and  Thomas~\cite{RST}. Despite receiving considerable attention over the years, Hadwiger's Conjecture remains wide open for all $t\ge 7$,  and is    considered among the most important problems in graph theory and has motivated numerous developments in graph coloring and graph minor theory.   
   K\"{u}hn  and Osthus~\cite{KuhOst03c} proved that Hadwiger's Conjecture is true for $C_4$-free graphs of sufficiently large chromatic number,  and for all graphs of girth at least $19$.     Until very recently  the best known upper bound on the chromatic number of graphs with no $K_t$ minor  is $O(t (\log t)^{1/2})$,  obtained independently by Kostochka~\cite{Kostochka82,Kostochka84} and Thomason~\cite{Thomason84}, while  Norin, Postle and the second  author~\cite{NPS20} improved    the frightening $(\log t)^{1/2}$ term to $(\log t)^{1/4}$. The current record 
 is $O(t\log \log t)$ due to
Delcourt and Postle~\cite{DelcourtPostle}.   \medskip

Given the notorious difficulty of Hadwiger's Conjecture, Paul Seymour in 2017 suggested the study of the following $H$-Hadwiger's Conjecture. 

\begin{conj}[$H$-Hadwiger's Conjecture]\label{HHC} For every graph $H$ on $t$ vertices, every graph with no $H$ minor is $(t-1)$-colorable. 
\end{conj}

 Jakobsen~\cite{Jakobsen71b} in 1971 proved that every graph with no $K_7^{-}$ minor is  $7$-colorable. It  is not   known yet whether  every graph with no $K_7$ minor is $7$-colorable;  some progress has been  made in \cite{RST22}.  For $H\in\{K_7^-, K_7^=, K_7^<\}$, proving that graphs with no $H$ minor are  $6$-colorable also remains open.   Kostochka~\cite{Kos14}  proved that   $H$-Hadwiger's Conjecture is true for graphs with no $K_{s,t}$ minor, provided that   $t>C(s\log s)^3$. Very recently, Norin and Seymour~\cite{NorSey22} proved that every graph on $n$ vertices with independence number two has an $H$ minor, where $H$  is a graph with  $\lceil n/2\rceil$ vertices and at least $  0.98688\cdot {{|H|}\choose2}-o(n^2)$ edges. We refer the reader to   a recent  paper of  the present  authors~\cite{K84} on partial results towards Hadwiger's Conjecture for $t\le 9$; and  recent surveys~\cite{CV2020, K2015,Seymoursurvey} for further  background on Hadwiger's Conjecture. \medskip

 Dirac  in 1964 began the study of a variation of $H$-Hadwiger's Conjecture in~\cite{Dirac64b} by  excluding  more than one forbidden minor simultaneously; he proved that every graph with no $\mathcal{K}_t^{-2}$ minor is  $(t-1)$-colorable for each $t\in\{5,6\}$. 
Jakobsen~\cite{Jakobsen71a} in 1971 proved that every graph with no $\mathcal{K}_7^{-2}$ minor is  $6$-colorable;   this implies that $H$-Hadwiger's Conjecture  is true  for all graphs $H$ on seven vertices such that $\Delta(\overline{H})\ge2$ and $\overline{H}$ has a matching of size two.  \medskip   
 
Very recently,  using the techniques developed in \cite{ KT05,RST} and  generalized Kempe chains of contraction-critical graphs    by Rolek  and the second author~\cite{RolekSong17a}, the present authors considered the case when $t=8$ and proved the following result. 

\begin{thm}[Lafferty and Song~\cite{K84}]\label{t:K84} Every graph with no $\mathcal{K}_8^{-4}$ minor is  $7$-colorable.  In particular,   $H$-Hadwiger's Conjecture  is true for all graphs $H$ on eight vertices such that  $\Delta(\overline{H})\ge4$,  and $\overline{H}$ has a perfect matching,  a triangle and a cycle of length four.
\end{thm}

The purpose of this paper is to  consider the next step and prove the following main result.

\begin{restatable}{thm}{main}\label{t:main} Every graph with no $\kns$ minor is $8$-colorable.     \end{restatable}

 \cref{t:main} implies that $H$-Hadwiger's Conjecture holds for all graphs $H$ on nine vertices such that $H$ is a subgraph of  every graph in $ \kns$.  Following the ideas in \cite{K84}, our proof of \cref{t:main} utilizes  an extremal function for $\kns$ minors (see \cref{t:exfun}),   generalized Kempe chains of contraction-critical graphs   (see \cref{l:wonderful}), and the method   for finding $\kns$ minors from three different   $K_6$ subgraphs   in $7$-connected graphs on at least $19$ vertices (see \cref{l:threek6s}).

\begin{restatable}{thm}{exfun}\label{t:exfun} Every graph on $n\ge 9$ vertices with at least $5n-14$ edges has a $\kns$ minor. \end{restatable} \medskip

\cref{t:exfun} is best possible in the sense that  every  $(K_8^=, 4)$-cockade on $n$ vertices has   $5n-14$ edges  but  no $\knf$ minor, where  for a graph $H$ and an integer $k\ge1$, an  $(H, k)$-cockade is defined recursively as follows: any graph isomorphic to $H$  
is an  $(H,k)$-cockade. Let $G_1$ and  $G_2$ be $(H, k)$-cockades and let $G$
be obtained from the disjoint union of $G_1$ and $G_2$ by identifying a clique
of size $k$ in $G_1$ with a clique of the same size in $G_2$. Then the graph
$G$ is also an $(H,k)$-cockade, and every $(H,k)$-cockade can be constructed
this way.  \medskip

This paper is organized as follows. In the next section, we introduce the necessary definitions and collect several tools which we will need later on.  We prove \cref{t:main} in Section~\ref{s:coloring}, and  \cref{t:exfun}  in Section~\ref{s:exfun}. 
 
 \section{Notation and tools}

 Let $G$ be a graph.   If $x,y$ are adjacent
vertices of   $G$, then we denote by $G/xy$ the graph obtained from $G$
by contracting the edge $xy$ and deleting all resulting parallel
edges. We simply write $G/e$ if $e=xy$.  If $u,v$ are distinct nonadjacent vertices of   $G$, then by
$G+uv$ we denote the graph obtained from $G$ by adding an edge
with ends $u$ and $v$.  If $u,v$ are adjacent or equal, then we define
$G+uv$ to be $G$.  Similarly, if  $M\subseteq E(G)\cup E(\overline{G})$, then   by
$G+M$ we denote the graph obtained from $G$ by adding  all the edges of $M$ to $G$.  Every edge in $\overline{G}$  is  called a \dfn{missing edge} of $G$. For a vertex $x\in V(G)$, we will use $N(x)$ to denote the set of vertices in $G$ which are adjacent to $x$.
We define $N[x] = N(x) \cup \{x\}$.  The degree of $x$ is denoted by $d_G(x)$ or
simply $d(x)$.   If  $A, B\subseteq V(G)$ are disjoint, we say that $A$ is \emph{complete} to $B$ if each vertex in $A$ is adjacent to all vertices in $B$, and $A$ is \emph{anticomplete} to $B$ if no vertex in $A$ is adjacent to any vertex in $B$.
If $A=\{a\}$, we simply say $a$ is complete to $B$ or $a$ is anticomplete to $B$. We use $e(A, B)$ to denote the number of edges between $A$ and $B$ in  $G$. 
The subgraph of $G$ induced by $A$, denoted by $G[A]$, is the graph with vertex set $A$ and edge set $\{xy \in E(G) \mid x, y \in A\}$. We denote by $B \less A$ the set $B - A$,   and $G \less A$ the subgraph of $G$ induced on $V(G) \less A$, respectively. 
If $A = \{a\}$, we simply write $B \less a$    and $G \less a$, respectively.
An $(A, B)$-path in $G$ is a path with one end in $A$ and  the other in $B$ such that  all its internal vertices lie  in $G\less (A\cup B)$. We simply say an $(a, B)$-path if $A=\{a\}$. It is worth noting that each vertex in $A \cap B$ is an $(A, B)$-path.
  For a positive integer $k$,  a $k$-vertex is a vertex of degree $k$, and a $k$-clique   is a set of $k$ pairwise adjacent vertices.     
 Let $\mathcal{F}$ be a family of graphs. A graph $G$ is \emph{$\mathcal{F}$-free} if it has no subgraph isomorphic to $H$ for every    $H\in\mathcal{F}$. We simply say $G$ is $H$-free if  $\mathcal{F}=\{H\}$.  The \dfn{join} $G+H$ (resp. \dfn{union} $G\cup H$) of two 
vertex-disjoint graphs
$G$ and $H$ is the graph having vertex set $V(G)\cup V(H)$  and edge set $E(G)
\cup E(H)\cup \{xy\, |\,  x\in V(G),  y\in V(H)\}$ (resp. $E(G)\cup E(H)$).  
We use the convention   ``$A:=$"  to mean that $A$ is defined to be
the right-hand side of the relation.
Finally, if $H$ is a connected subgraph of a graph $G$ and $y \in V(H)$, we say that we \textit{contract $H \less y$ onto $y$} when we contract $H$ to a single vertex, that is, contract all the edges of $H$.    \medskip

 To prove \cref{t:main}, we need to investigate the basic properties of contraction-critical graphs.  For a positive integer $k$, a graph $G$ is \dfn{$k$-contraction-critical} if $\chi(G)=k$ and every proper minor of $G$ is $(k-1)$-colorable.   
   Dirac~\cite{Dirac60} introduced the notion of contraction-critical graphs and proved \cref{l:alpha2} below; in the same paper he also proved that $5$-contraction-critical graphs are $5$-connected. The latter was then extended by Mader~\cite{7con} as stated in \cref{t:7conn}. It remains unknown whether every $k$-contraction-critical graph is $8$-connected for all $k\ge8$.

\begin{lem}[Dirac~\cite{Dirac60}]\label{l:alpha2}   Let $G$ be a  $k$-contraction-critical graph. Then for each  $v\in V(G)$, \[\alpha(G[N(v)])\le d(v)-k+2.\] 
\end{lem}

\begin{thm}[Mader~\cite{7con}]\label{t:7conn}  
For all $k \ge 7$, every $k$-contraction-critical graph is $7$-connected.
\end{thm}

  \cref{l:wonderful}  on contraction-critical graphs  turns out to be very powerful, as the existence of pairwise vertex-disjoint paths  is guaranteed without using the connectivity of such  graphs.    Recall that    every edge in $\overline{H}$ is a \dfn{missing edge} of a graph $H$.

   \begin{lem}[Rolek and Song~\cite{RolekSong17a}]\label{l:wonderful} 
Let $G$ be any $k$-contraction-critical graph. Let $x\in V(G)$ be a vertex of
     degree $k + s$ with $\alpha(G[N(x)]) = s + 2$ and let $S \subset N(x)$ with
    $ |S| = s + 2$ be any independent set, where $k \ge 4$ and $s \ge 0$ are integers.
     Let $M$ be a set of missing edges of $G[N(x) \setminus S]$.  Then there
     exists a collection $\{P_{uv}\mid uv\in M\} $ of paths in $G$ such that
     for each $uv\in M$, $P_{uv}$ has ends $u, v$ and all its internal vertices
     in $G \setminus N[x]$. Moreover,  if vertices $u,v,w,z$ with $uv,wz\in M$ are distinct, then
     the paths $P_{uv}$ and $P_{wz}$ are vertex-disjoint.
 \end{lem} 
   
 The proof of \cref{l:wonderful}  uses  Kempe chains.  Using a  result of Mader~\cite{7con} on rooted $K_4$ minors and the proof of \cref{l:wonderful}, the present authors~\cite{K84} proved a strengthened version of the remark  given in  \cite[Page 17]{RolekSong17a}. 
 
\begin{lem}[Lafferty and Song~\cite{K84}]\label{l:rootedK4} Let $G$ be any $k$-contraction-critical graph. Let $x\in V(G)$ be a vertex of
     degree $k + s$ with $\alpha(G[N(x)]) = s + 2$ and let $S \subset N(x)$ with
    $ |S| = s + 2$ be any independent set, where $k \ge 4$ and $s \ge 0$ are integers.
  If \[M=\{x_1y_1, x_1y_2, x_2y_1, x_2y_2, a_1b_{11}, \dots, a_1b_{1r_1},   \dots, a_mb_{m1}, \dots, a_mb_{mr_m}\}\] is a set of missing edges of $G[N(x)\less S]$, where  the vertices $x_1, x_2, y_1, y_2, a_1, \dots, a_m,
  b_{11}, \dots, b_{mr_m}\in N(x)\less S$ are all distinct, and  for all $1\le i  \le m$,  $a_ib_{i1}, \dots, a_ib_{ir_i}$ are $r_i$ missing edges  with $a_i$ as a common end, and $x_1x_2, y_1y_2\in E(G)$,   then  $G \se G[N[x]]+M$.    \end{lem}

\begin{rem}\label{r:rootedK4}  As observed in   \cite{K84},   \cref{l:rootedK4}  can be   applied      when  \[M=\{x_1y_1, x_1y_2, x_2y_1, x_2y_2, a_1b_{11}, \dots, a_1b_{1r_1},  \dots, a_mb_{m1}, \dots, a_mb_{mr_m}\}\] is a subset of  edges and missing edges of  $G[N(x)\less S]$, where    $x_1, x_2, y_1, y_2, a_1, \dots, a_m,
  b_{11}, \dots, b_{mr_m}\in N(x)\less S$ are all distinct, and $x_1x_2, y_1y_2\in E(G)$.  Under those circumstances, it suffices to apply \cref{l:rootedK4} to $M^*$, where $M^*= \{e\in M\mid e \text{ is a missing edge of } G[N(x)\less S])\}$. It is straightforward to see that  $G\se G[N[x]]+M$.  
  \end{rem}
   
\begin{figure}[htb]
\centering
\includegraphics[scale=0.5]{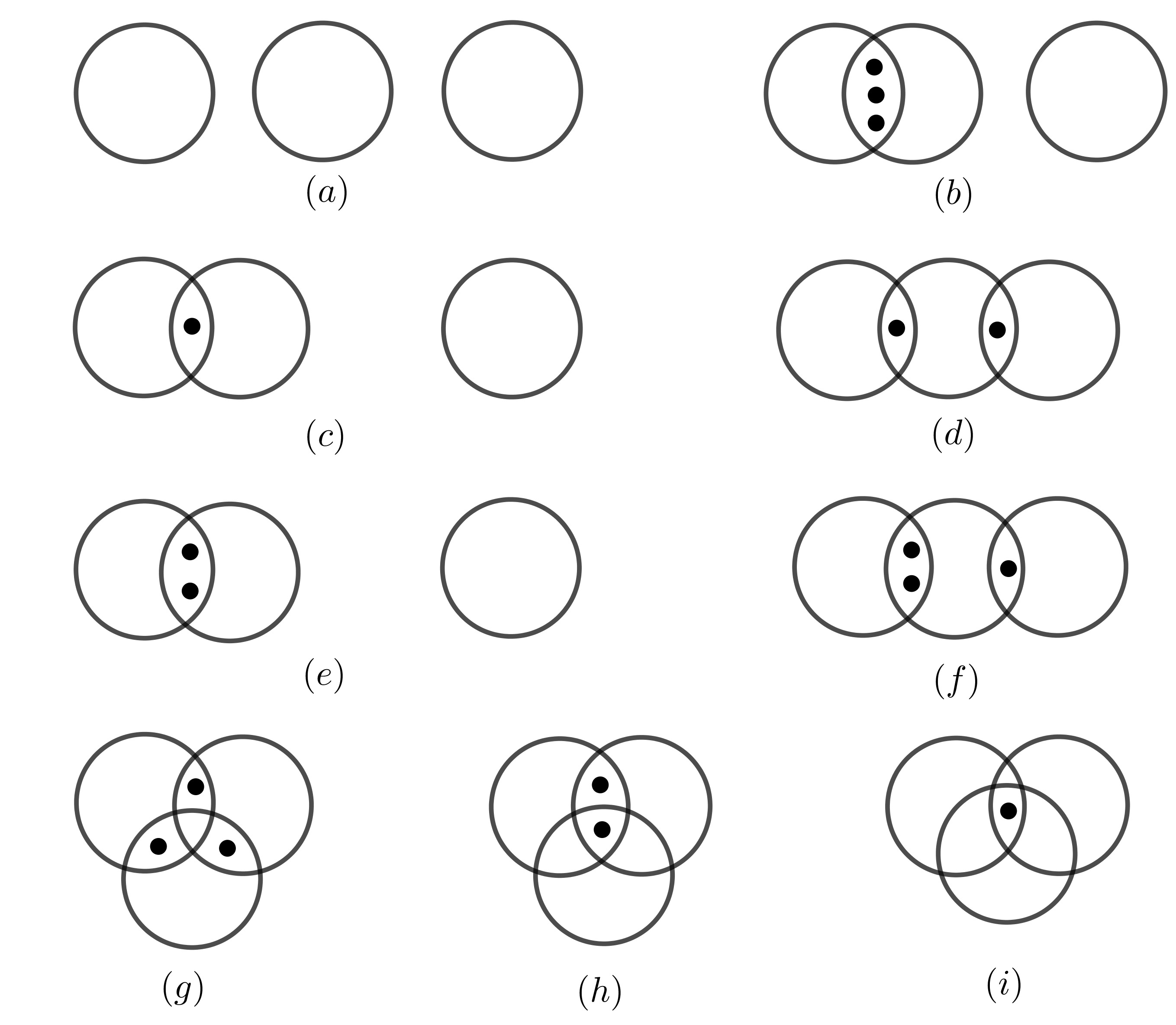}
\caption{The nine possibilities  for three $5$-cliques  in Theorem~\ref{t:goodpaths}.}
\label{fig:threek5s}
\end{figure}

Finally we need a tool to find a desired $\kns$ minor through three different $6$-cliques in   $7$-connected graphs. This method was first introduced  by Robertson, Seymour and Thomas~\cite{RST} to prove Hadwiger's Conjecture for $t=6$: they found a desired $K_6$ minor via three different $4$-cliques in $6$-connected non-apex graphs. The method  was later extended by Kawarabayashi and Toft~\cite{KT05} to find a desired $K_7$ minor via three different $5$-cliques in $7$-connected  graphs. 
It is worth noting that \cref{t:goodpaths} corresponds to \cite[Lemma 5]{KT05}, where the existence of such seven  ``good paths" follows from the proof  of   \cite[Lemma 5]{KT05}.

\begin{thm}[Kawarabayashi and Toft~\cite{KT05}]\label{t:goodpaths}
Let $G$ be a $7$-connected graph such that $|G| \geq 19$.
Let $L_1, L_2$, and $L_3$ be three different $5$-cliques of   $G$  such that  $|L_1\cup L_2\cup L_3|\ge 12$, that is, they fit into one of the nine configurations depicted in Figure~\ref{fig:threek5s}.
Then $G$ has seven pairwise vertex-disjoint ``good paths", where  a ``good path"  is  an $(L_i, L_j)$-path in $G$ with  $i \neq j$.
\end{thm}

\section{Coloring graphs with no $\kns$ minor}\label{s:coloring}

 We first use \cref{t:goodpaths} to prove a lemma that finds a $\kns$ minor via three different $6$-cliques in $7$-connected graphs with at least $19$ vertices. 

\begin{lem}
\label{l:threek6s}
Let $G$ be a $7$-connected graph such that  $|G|\ge19$. If  $L_1$, $L_2$, and $L_3$ are  three  $6$-cliques of $G$ satisfying   \[\min\{|L_1 \setminus (L_2 \cup L_3)|, |L_2 \setminus (L_1 \cup L_3)|, |L_3 \setminus (L_1 \cup L_2)| \} \geq 1,\tag{$*$}\] 
then $G$ has a $\kns$ minor.
\end{lem}

\begin{proof} Suppose $G$ has no  $\kns$ minor. By the assumption ($*$),  we see that $|L_1 \cap L_2\cap L_3| \le 5$, and $|L_i \cap L_j| \le 5$ for $1\le i<j\le 3$. 
We first observe that  $G$ is $\km{8}{5}$-free:   suppose $G$ has an $H$ subgraph for some $H\in \km{8}{5}$. Since $G$ is $7$-connected, we see that there are at least seven pairwise disjoint $(V(H), V(C))$-paths in $G$ for every component $C$ of  $G \setminus  V(H)$. Thus we obtain a $\km{9}{6}$ minor by  contracting a  component of $G \setminus  V(H)$ to a single vertex, a contradiction.  It follows   that $|L_1 \cap L_2\cap L_3| \ne 5$, and $|L_i \cap L_j| \ne  4$ for $1\le i<j\le 3$, else neither $G[L_1 \cup L_2\cup L_3]$ nor  $G[L_i \cup L_j]$ is $\km{8}{5}$-free. We may assume that $|L_1 \cap L_2| \ge |L_1 \cap L_3|  \ge |L_2 \cap L_3|$. \medskip

Suppose first $|L_1 \cap L_2| \le 1$. Let $L_i'$ be a $5$-clique of $L_i$ for each $i\in[3]$ such that $L_i\cap L_j=L_i'\cap L_j'$ for $1\le i<j\le 3$. Then  $L_1', L_2'$ and $ L_3'$   fit into one of the five configurations  in Figure~\ref{fig:threek5s}(a,c,d,g,i). By Theorem~\ref{t:goodpaths} applied to $L_1', L_2'$ and $ L_3'$,  there exist seven pairwise  vertex-disjoint ``good paths", say $Q_1, \ldots, Q_7$,  between   $L_1, L_2$, and $L_3$;    we choose $Q_1, \ldots, Q_7$ so that   $|V(Q_1)|+ \cdots+|V(Q_7)|$ is as small as possible. It follows that  no internal vertex of each $Q_i$ belongs to   $L_1 \cup L_2 \cup L_3$,  and no   vertex of $L_i\cap L_j$ belongs to a ``good path''  of length at least one. Let $t_{i,j}$ denote the number of ``good paths" between  $L_i$ and $L_j$ for $1\le i<j\le 3$. We may assume that $t_{1,2}\ge t_{1,3}\ge t_{2,3}$. Then $3\le t_{1,2}\le 5$. We may further  assume that $Q_6$ and $Q_7$ are $(L_1, L_2)$-paths  of length at least one.  Suppose $t_{1,2}= 5$.   By    contracting each of $Q_1, \ldots, Q_5$ to a single vertex, all the edges,  but one,  of $Q_6$ (that is, contracting $Q_6$ to a $K_2$), and all the edges,  but one, of $Q_7$ (that is, contracting $Q_7$ to a $K_2$),  we see that $G\se \kns$, a contradiction. Thus  $3\le t_{1,2}\le 4$.  Recall that $ t_{1,3}\ge  t_{2,3}$.   
 Let  $x\in   L_2$; in addition, let    $y\in     L_1\cup L_2$ with $y\ne x$ when  $t_{1,2}=3$,   such that neither $x$ nor $y$ is an end of any ``good path".  But now  contracting each of $Q_1, \ldots, Q_7$ to a single vertex, together with $x$ and $y$,  yields a   $ \kns$ minor in $G$ when $t_{1,2}=3$, and contracting each of $Q_1, \ldots, Q_6$ to a single vertex and $Q_7$ to a $K_2$, together with $x$,  yields a   $ \kns$ minor in $G$ when $t_{1,2}=4$,   a contradiction.
 This proves that $|L_1 \cap L_2| \ge 2$. Let $a_1, \ldots, a_p\in L_1\cap L_2$, where $p:=|L_1\cap L_2|$. Then $p=5$ or $2\le p\le 3$. \medskip

Suppose next $2\le p\le  3$.
 By Menger's Theorem, there exist  $6-p\ge3$  pairwise  vertex-disjoint $(L_1 \setminus L_2, L_2 \setminus L_1)$-paths, say $Q_1, \ldots, Q_{6-p}$,  in $G \setminus \{ a_1, \ldots, a_p\}$.  But then    we obtain a $\kns$ minor in $G$ from $G[L_1\cup L_2]$ by contracting each of $Q_1, Q_2, Q_3$ to   a $K_2$; in addition,   contracting  $Q_4$ to a single vertex  when $p=2$, a contraction.  \medskip

It remains to consider the case $p=5$.  Let $x \in L_1 \setminus L_2 $ and   $ y \in L_2 \setminus L_1 $. By the assumption ($*$), $x, y\notin L_3$.  Let  $z\in L_3\less (L_1\cup L_2)$.  Since $G$ is $\km{8}{5}$-free, we see that  $|L_1 \cap L_2\cap L_3|\le 2$, else $G[L_1\cup L_2\cup\{z\}]$ is not $\km{8}{5}$-free.     Suppose $  |L_1 \cap L_2\cap L_3|= 2$. We may assume $a_4, a_5\in  L_1 \cap L_2\cap L_3 $. Let $z'\in L_3\less(L_1\cup L_2)$ such that $z'\ne z$.  Then $G\less \{a_4, a_5\}$ has five pairwise internally vertex-disjoint $(z, \{x, a_1, a_2, a_3,  y\})$-paths, say $Q_1, \ldots, Q_5$. We may assume that $z'$ does not belong to  $Q_1, \ldots, Q_4$. Let $Q_5^*$ be the $(z', w)$-subpath of $Q_5$ when $z'$ lies on $Q_5$, where $w$ is the other end of $Q_5$. Then $G\se \kns$ from $G[L_1\cup L_2\cup\{z,z'\}]$ by contracting each  of 
$Q_1\less z, \ldots,  Q_5\less z$ to a single vertex  when $z'\notin V(Q_5)$; and  each of $Q_1\less z, \ldots,  Q_4\less z $, and  $Q_5^* \less z'$  to a single vertex when $z'\in V(Q_5)$, a contradiction. 
This proves that  $  |L_1 \cap L_2\cap L_3|\le 1$. By Menger's Theorem, $G\less y$ has six pairwise vertex-disjoint   $(L_3, L_1)$-paths, say  $Q_1, \dotsc, Q_6$. We may assume that    $a_i$ is an end of $Q_i$   for each $i \in [5]$.  Then $x$ is an end of $Q_6$. We may assume further assume that $a_5\notin L_3$. But then we obtain a $\km{9}{5}$ minor in $G$ from $G[L_1\cup L_2\cup L_3]$ by contracting each   of $Q_1, \ldots, Q_4, Q_5\less a_5, Q_6\less x$ to a  single vertex, a contradiction.\medskip

This completes the proof of \cref{l:threek6s}.
\end{proof}

\begin{figure}[htb]
\centering
\includegraphics[scale=0.16]{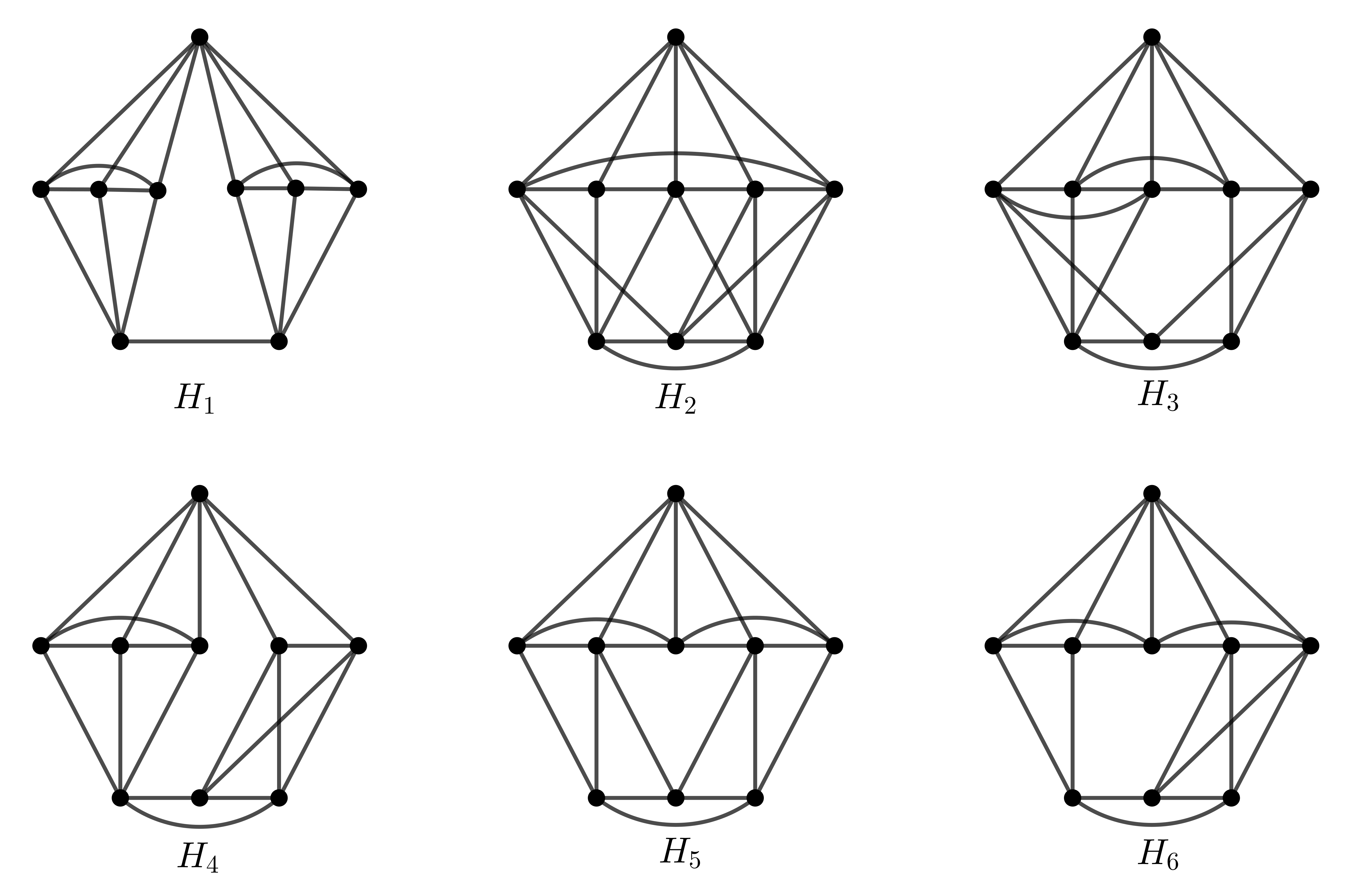}
\caption{Six $K_5$-free  graphs $H$ with  $|H|=9$ and $\alpha(H)=2$.}
\label{fig:counter}
\end{figure}

\begin{lem}
\label{l:H9}
Let $H$ be a graph such that $|H| = 9$ and $\alpha(H) = 2$.
Then  $H$  contains   $K_5$  or  one of the graphs in Figure~\ref{fig:counter} as a spanning subgraph. 
\end{lem}
\begin{proof}
 Suppose $H$ is $K_5$-free, and  $H_i$-free for each   $H_i$ given  in Figure~\ref{fig:counter}.   We may assume that $H$ is edge-minimal subject to being $K_5$-free and $\alpha(H)=2$.   Then $H$ has no dominating edge, where an edge $xy\in E(H)$ is dominating if every vertex in $V(H)\less\{x, y\}$ is adjacent to $x$ or $y$.   This implies that $\Delta(H)\le7$ and  \medskip
 
 \noindent (a) no vertex in $N(v)$ is complete to $V(H)\less N[v]$ for each $v\in V(H)$. \medskip
 
  Since $\alpha(H)=2$, we see that,    for each $v\in V(H)$, $V(H)\less N[v]$ is a clique, and so  $|H\backslash N[v]|\le4$ because $H$ is $K_5$-free.   Then   $\delta(H) \ge 4$ and  $H[N(v)]$ is $K_4$-free for each $v\in V(H)$.   By (a),   $\Delta(H)\le 6$. 
 Let $x\in V(H)$ be a vertex of   degree   $\Delta(H)$.  Let  $N(x):=\{x_1, \ldots, x_{d(x)}\}$ and $V(H)\less N[x]:=\{y_1, \ldots, y_{8-d(x)}\}$.     Suppose $d(x)=4$. Then  $V(H)\less N[x]$ is a $4$-clique,  and each $y_i$ is adjacent to exactly one  vertex  in $N(x)$.  We may assume that $x_1y_1\in E(H)$. We may further assume that $x_1y_4\notin E(H)$ and $x_4y_4\in E(H)$. Then $\{x_2, x_3, x_4\}$ is a $3$-clique, and $x_1$ is complete to $\{x_2, x_3\}$ because $y_4$ is anticomplete to $\{x_2, x_3\}$. But then $x_4$ is complete to $\{y_2, y_3\}$ because $x_1$ is anticomplete to $\{x_4, y_2,y_3\}$, contrary to the fact that $\Delta(H)=4$.     Suppose next $d(x)=6$.   Then  $V(H)\less N[x]=\{y_1, y_2\}$, and by (a), both $y_1$ and $y_2$ are $4$-vertices such that $y_1$ and $y_2$ have no common neighbor in $N(x)$. Thus $N(y_1)\cap N(x)$ and $N(y_2)\cap N(x)$ are disjoint $3$-cliques in $H$.   But then $H$ contains $H_1$ as a subgraph, a contradiction.  This proves that   $d(x)=5$, and so $N(x)=\{x_1, \ldots, x_5\}$ and $V(H)\less N[x]=\{y_1, y_2, y_3\}$.   \medskip

  Suppose $H[N(x)]$ is $K_3$-free. Note that   $\alpha(H[N(x)])=2$.  Thus $H[N(x)]=C_5$, say with vertices $x_1, \ldots, x_5$ in order.   Since $d(y_1)\le 5$, we may assume that $y_1$ is anticomplete to $\{x_4, x_5\}$. Then $y_1$ is complete to $\{x_1, x_2, x_3\}$.  By (a) applied to $\{x_1, x_2, x_3\}$, we may further assume that $y_3$ is anticomplete to $\{x_1, x_2\}$. Then $y_3$ is complete to $\{x_3, x_4, x_5\}$ and so $y_2x_3\notin E(H)$. Then $y_2$ is complete to $\{x_1, x_5\}$; in addition, $y_2$ is adjacent to exactly one of $x_2$ and  $x_4$ because $\alpha(H)=2$ and $\Delta(H)=5$.  It follows that  $H$ contains $H_2$ as a subgraph, a contradiction. This proves that    \medskip

\noindent (b)   $H[N(v)]$  contains  $K_3$ as a  subgraph  for every  $5$-vertex $v$.  \medskip

 Note that $\delta(H[N(x)])\ge 1$ because $H$ is $K_5$-free. Suppose $\delta(H[N(x)])= 1$. We may assume that $x_5x_4\in E(H)$ and $x_5$ is anticomplete to $\{x_1, x_2, x_3\}$.  By (a), we may assume that  $x_5y_1\notin E(H)$. Then $x_5$ is a $4$-vertex, $\{y_1, x_1, x_2, x_3\}$ is a $4$-clique, and $x_5$ is complete to $\{y_2, y_3\}$.  By (a) applied to $\{x_1, x_2, x_3\}$, we may assume that  $y_2$ is anticomplete to $\{x_2, x_3\}$. Suppose $x_4y_2\notin E(H)$. Then $x_4$ is complete to $\{x_2, x_3\}$ and $y_2x_1, x_4y_3\in E(H)$. Thus $H$ contains $H_3$ as a subgraph, a contradiction.  It follows that $x_4y_2\in E(H)$. Then $x_4y_3\notin E(H)$, else $H$ contains $H_4$ as a subgraph. Note that each of $x_1, x_2, x_3$ is adjacent to exactly one of $x_4$ and $y_3$; and either $e_H(\{x_1, x_2, x_3\}, x_4)=2$ and $e_H(\{x_1, x_2, x_3\}, y_3)=1$, or $e_H(\{x_1, x_2, x_3\}, x_4)=1$ and $e_H(\{x_1, x_2, x_3\}, y_3)=2$. In the former case,  we may assume that $x_1y_3, x_2x_4, x_3x_4\in E(H)$; thus $H$ contains $H_3$ as a subgraph, a contradiction.  In the latter case, we may assume that $x_1y_3, x_2y_3, x_3x_4\in E(H)$; again  $H$ contains $H_3$ as a subgraph by drawing the graph $H$ according to the $5$-vertex  $y_1$, a contradiction. This proves that \medskip
 
  \noindent (c)\, $\delta(H[N(v)])\ge 2$ for every $5$-vertex $v$. \medskip
 
 By (b), we may assume that $\{x_1, x_2, x_3\}$ is a clique.  Suppose $x_4x_5\notin E(H)$. Note that neither $x_4$ nor $x_5$ is complete to $\{x_1, x_2, x_3\}$, and no vertex in $\{x_1, x_2, x_3\}$ is   anticomplete to $\{ x_4, x_5\}$. By (c),  we may assume that $x_4$ is complete to $\{x_1, x_2\}$ and $x_5$ is complete to $\{x_2, x_3\}$. Then $x_1x_5, x_3x_4\notin E(H)$. Note that each $y_j$ is adjacent to at least two vertices in $N(x)$ for each $j\in[3]$. However,  each of $x_1$ and $x_3$ is adjacent to at most one vertex in $\{y_1, y_2, y_3\}$;  each of $x_4$ and $x_5$ is adjacent to at most two  vertices in $\{y_1, y_2, y_3\}$; and  $x_2$ is  anticomplete to  $\{y_1, y_2, y_3\}$.  It follows that $e_H(\{y_1, y_2, y_3\}, N(x))=6$;   and  every $x_i$ is a $5$-vertex and every $y_j$ is a $4$-vertex for each $i\in[5]$ and $j\in[3]$.  We may assume that $x_4$ is complete to $\{y_1, y_2\}$. Then $x_4y_3\notin E(H)$ and so $y_3$ is complete to $\{x_3, x_5\}$.  But then $y_3$ is adjacent to $x_5$ only in 
$H[N(x_3)]$, contrary to (c). We may assume that  \medskip

  \noindent (d)\,  $H[N(v)]$  is $K_3\cup \overline{K}_2$-free for every $5$-vertex $v$.\medskip

 It remains to consider the case 
 $x_4x_5\in E(H)$.   Suppose  $x_i$ is complete to $\{x_4, x_5\}$ for some $i\in[3]$, say $i=3$. We may assume that $x_1x_4\notin E(H)$ because $H[N(x)]$ is $K_4$-free.  By (d) applied to $H[N(x)]$,  we have $x_2x_5\notin E(H)$; in addition, either  $x_1x_5, x_2x_4\notin E(H) $, or $x_1x_5, x_2x_4\in E(H) $.  Since $e_H(\{y_1, y_2, y_3\}, N(x))\ge 6$, we see that  $\{x_1, x_2\}$ is anticomplete to $\{x_4, x_5\}$.  By (a), we may assume that $x_1y_3, x_4y_j\notin E(H)$ for some $j\in[3]$. Then $y_3$ is complete to $\{x_4, x_5\}$. Thus $j\ne 3$. We may assume that $j=1$. Then $y_1$ is complete to $\{x_1, x_2\}$.  Since $H$ is $H_5$-free, we see that $y_2$ is not complete to $\{x_2, x_4\}$. We may assume that    $x_2y_2\notin E(H)$.  Then $y_2$ is complete to $\{x_4, x_5\}$ because $\{x_1, x_2\}$ is anticomplete to $\{x_4, x_5\}$. But then   $H$ contains  $H_6$ as a subgraph, a contradiction.   This proves that no $x_i$ is  complete to $\{x_4, x_5\}$ for each $i\in[3]$. By (c), we may assume that $x_2x_4, x_3x_5\in E(H)$. Then $x_2x_5, x_3x_4\notin E(H)$. By (a), we may assume that $x_5y_3\notin E(H)$. Then $y_3x_2\in E(H)$.  By (d) applied to $H[N(x_2)]$, we have $y_3x_4\in E(H)$. By (a), we may assume that $x_4y_2\notin E(H)$. Then $x_4$ is anticomplete to $\{x_3, y_2\}$, and so $x_3y_2\in E(H)$. Thus $y_1$ is anticomplete to $\{x_2, x_3\}$, and so $y_1$ is complete to $\{x_4, x_5\}$.    But then $x_4$ is a $5$-vertex such that $G[N(x_4)]=C_5$, contrary to (b).  \medskip
 
 This completes the proof of \cref{l:computer}.
 \end{proof}
 
 \begin{figure}[htb]
\centering
\includegraphics[scale=0.16]{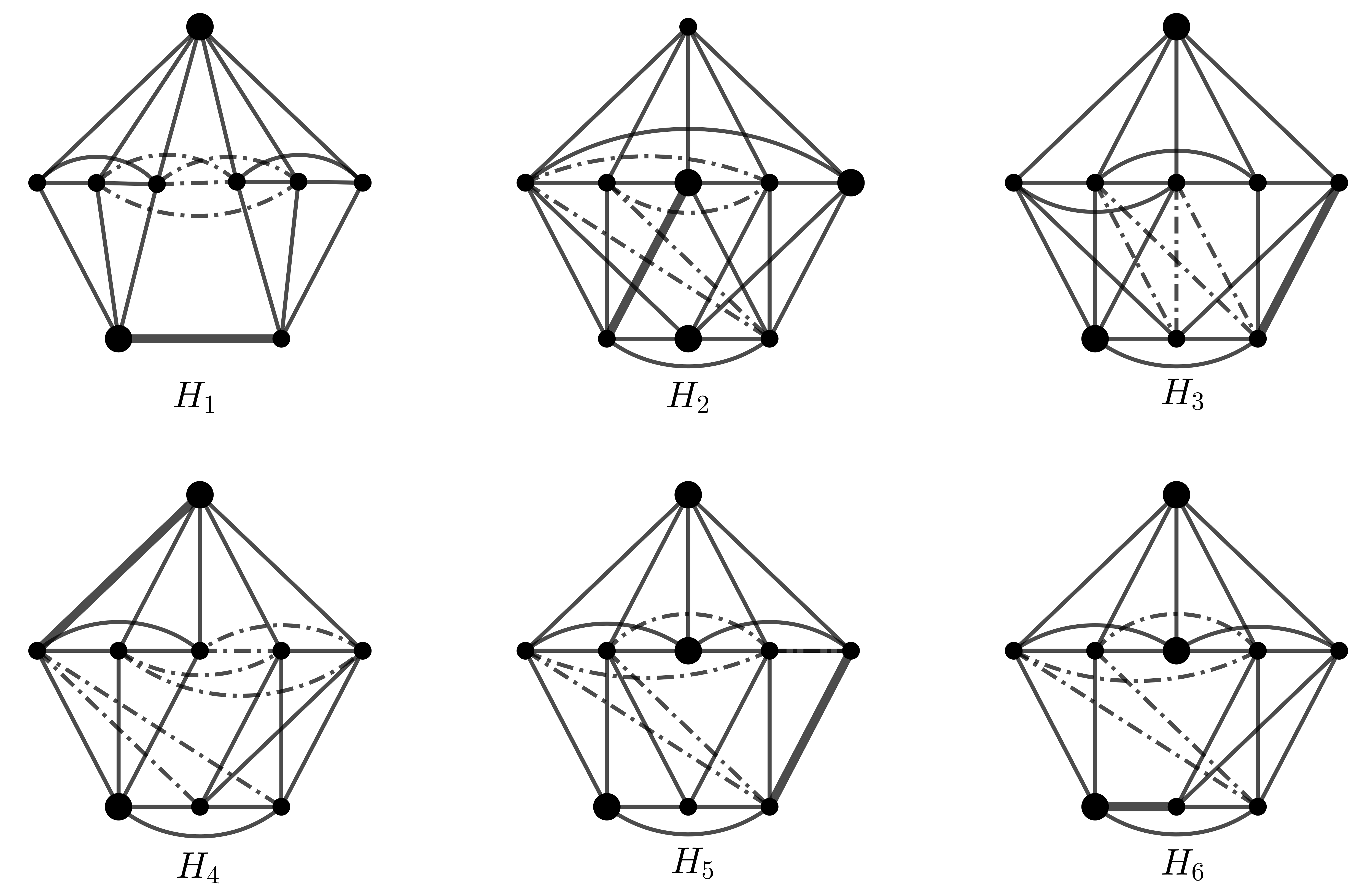}
\caption{Graphs in \cref{fig:counter} with bold vertices and  edges depicted, and dashed edges added.}
\label{fig:wonderful}
\end{figure}

We are now ready to prove Theorem~\ref{t:main}, which we restate for convenience.\main* 

\begin{proof} Suppose the assertion is false. Let $G$ be a  graph with no $\kns$ minor such  that $\chi(G) \ge 9$.  
We may choose such a  graph $G$ so that it is $9$-contraction-critical. Then $\delta(G)\ge8$,    $G$ is $7$-connected  by  \cref{t:7conn}, and     $\delta(G) \le 9$ by \cref{t:exfun}. 
Let $x\in V(G)$ be of minimum degree. Since $G$ is $9$-contraction-critical and has no $\kns$ minor, by  \cref{l:alpha2} applied to $G[N(x)]$, we see that    $\delta(G)=9$ and $\alpha(G[N(x)]) = 2$.  We next prove that $G[N(x)]$ contains a  $5$-clique.
Suppose $G[N(x)]$ is $K_5$-free. By \cref{l:H9},   $G[N(x)]$ contains   a  spanning subgraph isomorphic to one of the graphs in Figure~\ref{fig:counter}.  Let $A$ be the set of all  bold vertices,    $M$ be the set of all dashed edges and $e$ the bold edge  in  each  $H_i$  given in  Figure~\ref{fig:wonderful}.  Since $G[N(x)]$ is $K_5$-free, we see that $A$ is not a clique in $G$. Let $S$ be a set of two nonadjacent vertices in $A$. Note that no vertex in $S$ is incident with any dashed edges in $M$.   By   \cref{l:rootedK4}  applied to  $G[N(x)]$ with $S$ and $M$ given above,  we see that $G[N[x]]+M\se \kns$ by contracting the bold edge  $e$, a contradiction.  This proves that $G[N(x)]$ contains a  $5$-clique for all such $9$-vertices $x$ in $G$.  Let  $n_9$ denote  the number of $9$-vertices   in $G$.   Then $e(G)\ge \big(9n_9+10(|G|-n_9)\big)/2=(10|G|-n_9)/2 $. By \cref{t:exfun}, $5|G|-15\ge e(G)\ge (10|G|-n_9)/2$. It follows that $n_9\ge30$. Then $G$ contains at least  three   pairwise nonadjacent    $9$-vertices in $G$.   Let   $x_1, x_2, x_3\in V(G) $  be  three   pairwise nonadjacent    $9$-vertices in $G$.  For each $i\in[3]$, $G[N(x_i)]$ has a $5$-clique;   let $L_i $ be a $6$-clique  of $G[N[x_i]]$ such that $x_i\in L_i $.    Then 
 \[\min\{|L_1 \setminus (L_2 \cup L_3)|, |L_2 \setminus (L_1 \cup L_3)|, |L_3 \setminus (L_1 \cup L_2)| \} \geq 1.\]
 By  \cref{l:threek6s},    $G\se\kns$, a contradiction.   \medskip

This   completes the proof of  \cref{t:main}.  
 \end{proof}

\section{An extremal function for $\kns$ minors}\label{s:exfun}

 Throughout this section, if  $G$ is a graph and $K$ is a subgraph of $G$, then by $N(K)$ we denote
the set of vertices of $V(G)\less V(K)$ that are adjacent to a vertex of $K$.
 If $V(K)=\{x\}$, then  $N(K)=N(x)$. 
It can be easily checked that for each vertex  $x\in V(G)$, if $K$ is a component of $G\less N[x]$, then $N(K)$ is a minimal separating set of $G$. \medskip

Lemma~\ref{l:k4-} follows from the proof of Lemma~16 of J\o rgensen~\cite{Jor01}. A proof can   be found in~\cite{K84}. 

\begin{lem}[J\o rgensen~\cite{Jor01}]\label{l:k4-} Let $G$ be a $4$-connected graph and let $S\subseteq V(G)$  be a separating set of four vertices. Let $G_1$ and $G_2$ be proper subgraphs of G so that $G_1\cup G_2=G$ and $G_1\cap G_2=G[S]$.   Let $d_1$ be the 
largest integer so that $G_1$ contains pairwise disjoint sets of vertices 
$V_1,  V_2, V_3, V_4$ so that $G_1[V_j]$ is connected, 
 $|S\cap V_j|=1$ for  $1\le j\le  4$, and so that the graph obtained from $G_1$ by contracting each of $G_1[V_1],   G_1[V_2],G_1[V_3],G_1[V_4]$ to a single vertex 
and deleting $V(G)\less \bigcup_{j=1}^4 V_j$ has 
$e(G[S])+d_1$ edges.    If $|G_1|\ge 6$, then   \[e(G[S]) + d_1 \geq 5.\]
\end{lem}

We next prove  a lemma that will be needed in the proof of Theorem~\ref{t:exfun}. 

\begin{lem}
\label{l:computer}
Let $H$ be a graph  on eight or nine vertices. If  $\delta(H) \geq 5$, then
  $H$ has a vertex $v$ such that $H \setminus v$ has a $ \km{7}{6}$ minor.
\end{lem}
\begin{proof}    We may assume that  $\delta(H) =5$ and every edge   is incident with a $5$-vertex in $H$.   Suppose $H \setminus v $ has no  $\km{7}{6}$ minor for every $v\in V(H)$.  Then  $|H|=9$, else for any   $5$-vertex $v$ in $H$,     $e(H\less v)\ge 20-5=e(K_7)-6$, and so $H\less v$ has a $\km{7}{6}$ minor,  a contradiction.     We claim that  some  edge in $H$ belongs to at most two triangles. Suppose not. Then every edge in $H$ belongs to at least three triangles. Let $x$ be a $5$-vertex in $H$. Then $\delta(H[N(x)])\ge3$, and so $H[N(x)]$ contains $K_5^=$ as a spanning subgraph; in addition, every vertex in $H\less N[x]$ is adjacent to at least three vertices in $N(x)$. It follows that $H[N[x]\cup\{y\}]\se \km{7}{5}$, where $y\in V(H)\less N[x]$, a contradiction.  Thus there exists an edge $uw\in E(H)$ such that $uw$ belongs at most two triangles. Let $H^*:=H/uw$. Then $e(H^*)\ge e(H)-3\ge 23-3=20$. Note that $|H^*|=8$. Similar to the case when $|H|=8$,   we see that  if $\delta(H^*)\ge5$, then $H^*\less v$ has a $\km{7}{6}$ minor for any   $5$-vertex $v$ in $H$. Thus $\delta(H^*)\le 4$. 
Let $v\in V(H^*)$ such that $d_{H^*}(v)\le4$.  Then  $ e(H^*\less v)\ge 20-4=e(K_7)-5$, and so  $  H^*\less v$ has a  $\km{7}{5}$ minor, a contradiction. \end{proof}

We are now ready to prove Theorem~\ref{t:exfun}, which we restate for convenience. \exfun*

 \begin{proof} Suppose the assertion is false. Let $G$ be a graph on   $n \geq 9$ vertices with  $ e(G) \geq  5n-14$ and, subject to this, $n$ is minimum. We may assume that $e(G) = 5n-14$.
It is straightforward  to check that $G\se  \kns$ when $n=9$.  Thus  $n \geq 10$.  We next prove several claims. \medskip

\setcounter{counter}{0}
\noindent {\bf Claim\refstepcounter{counter} \label{c:d5} \arabic{counter}.}  
 $\delta(G) \geq 6$.
 
\begin{proof} Suppose $\delta(G)\le 5$.  Let $x\in V(G)$  with $d(x)\le 5$. Then 
	\[e(G \setminus x)  =e(G)-d(x)\ge (5n-14)-5   =5(n-1)-14.\] 
	Thus  $G \setminus x $ has a $ \kns$ minor by the minimality of $G$,   a contradiction.
\end{proof}

\noindent {\bf Claim\refstepcounter{counter}\label{c:triangles} \arabic{counter}.}  
Every edge in $G$ belongs to at least five triangles. Moreover, $G[N[x]]=K_7$  if $x\in V(G)$ is a $6$-vertex,  and $G[N[x]]$ contains a $K_8^\equiv$  subgraph if $x\in V(G)$ is a $7$-vertex.
 
\begin{proof}
 Suppose there exists   an edge  $uv \in E(G)$ such that  $uv$ belongs to at most four  triangles. Then
	\[e(G / uv) \ge (5n-14)-5   =5|G/uv| -14.\]  
	Thus $G /uv\se \kns$ by the minimality of $G$, a contradiction.  Since every edge in $G$ belongs to at least five triangles, we see that  $G[N[x]]=K_7$ for each $6$-vertex $x$  in $G$, and $G[N[y]]$ contains a $K_8^\equiv$  subgraph for each $7$-vertex $y$  in $G$.  \end{proof}
 
  \noindent {\bf Claim\refstepcounter{counter}\label{c:n12} \arabic{counter}.}  
$n\ge12$.
 
\begin{proof}
Suppose $10\le n\le 11$. Let $x\in V(G)$ be a vertex of degree $\delta(G)$. Then $d(x)\le 7$ because $e(G) = 5n-14$. By Claim~\ref{c:d5}, $6\le d(x)\le 7$.
Suppose $ d(x)=7$.  Then    $G[N[x]]$ contains a $K_8^\equiv$ subgraph by Claim~\ref{c:triangles}, and every vertex  in $  V(G)\less N[x]$ is adjacent to at least  five  vertices in $N(x)$.  But then   $G\se G[\{v\}\cup N[x]] \se\kns$ for each $  v\in V(G)\less N[x]$, a contradiction.  Thus  $d(x)=6$.  Then $G[N[x]]=K_7$ by Claim~\ref{c:triangles}. Let $y, z\in V(G)\less N[x]$. Then $e_G(\{y, z\}, N(x))\ge 2(6-(n-8))=28-2n$. If   $e_G(\{y, z\}, N(x))\ge 9$, or $e_G(\{y, z\}, N(x))\ge 8$ and $yz\in E(G)$,  then    $G[\{y, z\}\cup N[x]]\se \kns$, a contradiction. Thus $  e_G(\{y, z\}, N(x)) =8 $ and $yz\notin E(G)$, or $6\le  e_G(\{y, z\}, N(x)) \le 7 $.  In the former case, there exists $w\in V(G)\less N[x]$ such that $w$ is complete to $\{y,z\}$. But then $G[\{y, z,w\}\cup N[x]]/wz\se \kns$. Thus $6\le  e_G(\{y, z\}, N(x)) \le 7 $ for any two vertices $y, z\in V(G)\less N[x]$. It follows that $n=11$,    $V(G)\less N[x]$ is a $4$-clique,  no vertex of $V(G)\less N[x]$ has degree at least eight, and  at least three vertices of $V(G)\less N[x]$ are $6$-vertices in $G$. Thus $e_G(V(G)\less N[x], N(x))\le 12+1=13$. But then \[e(G)=e(G[N[x]])+e_G(V(G)\less N[x], N(x))+e(G\less N[x])\le 21+13+6< 5\times 11-14,\]  which is impossible. \end{proof}

\noindent {\bf Claim\refstepcounter{counter}\label{c:55} \arabic{counter}.}  
No three  $6$-vertices  in $G$ are pairwise adjacent.
 
\begin{proof} 
Suppose  there exist three distinct $6$ vertices, say $x, y, z$, in $G$ such that $\{x,y,z\}$ is a $3$-clique. 
Then $|G \setminus \{ x, y,z \}|=n-3 \geq 9$   by Claim~\ref{c:n12},   and  \[ e(G \setminus \{ x, y,z \}) = e(G)-15=     5(n-3) - 14.\] 
 Thus $G \setminus \{ x, y,z \}$ has a  $\kns$ by the minimality of $G$, a contradiction.  
 \end{proof}

Let $S$ be a minimal separating set of vertices in $G$, and let $G_1$ and $G_2$ be proper subgraphs of $G$ so 
that $G=G_1\cup G_2$ and $G_1\cap G_2=G[S]$. 
For each $i\in[2]$,   let $d_i$ be the 
largest integer so that $G_i$ contains pairwise disjoint sets of vertices 
$V_1,   \dots, V_p$ so that $G_i[V_j]$ is connected, 
 $|S\cap V_j|=1$ for  $1\le j\le p :=|S|$, and so that the graph obtained from $G_i$ by contracting each of $G_i[V_1],   \dots, G_i[V_p]$ to a single vertex 
and deleting $V(G)\less \bigcup_{j=1}^p V_j$ has 
$e(G[S])+d_i$ edges.   It follows from the minimality of $G$ that for each $i\in[2]$, 
\[ e(G_i) + d_{3-i} \le 5|G_1| - 15 \,\,  \text{ if } |G_i|\ge9.\tag{$ \bigstar$}\]    
 
\noindent {\bf Claim\refstepcounter{counter}\label{c:nodis7} \arabic{counter}.}   
If $|G_i|=8$ for some $i\in [2]$, then    $|S|\le 4$ and     some vertex in $V(G_i)\less S$ is a $7$-vertex in $G$.  
\begin{proof}
Suppose, say,  $|G_1| = 8$.    Let $C$ be a component of $G_2 \setminus S$.  We first prove that  $G_1\less S$ is  connected. Suppose not. By Claims~\ref{c:d5} and \ref{c:triangles},    $G_1\less S$ must contain two nonadjacent $6$-vertices in $G$ with   $G[S]=K_6$. Thus      $G_1=K_8^-$ and so  $G\se \km{9}{3}$ by contracting   $C$   to  a single vertex, a contradiction.   Thus    $G_1\less S$ is  connected. We next prove that some vertex  in $V(G_1)\less S$ is a   $7$-vertex in $G$. Suppose no vertex in $V(G_1)\less S$ is a  $7$-vertex in $G$. Then every vertex in $V(G_1)\less S$ is a  $6$-vertex in $G$.    By  Claim~\ref{c:triangles} and the fact that $G_1\less S$ is connected, we see that $V(G_1)\less S$ is a clique of order $8-|S| $. By Claim~\ref{c:55}, $|V(G_1)\less S|\le 2$. By the minimality of $S$, every vertex in $S$ is adjacent to at least one vertex in $V(G_1)\less S$. It follows that   $|V(G_1)\less S|= 2$  and $G_1=K_8^-$. But then $G\se\km{9}{2}$ by contracting   $C$   to  a single vertex, a contradiction.  Finally, let $x\in V(G_1)\less S$ be a $7$-vertex in $G$. By  Claim~\ref{c:triangles}, $G_1=G[N[x]]$ contains $K_8^\equiv$ as a spanning subgraph. Thus $|S|\le 4$, else $G\se\kns$ by contracting $C$ to a single vertex.  \end{proof}

\noindent {\bf Claim\refstepcounter{counter}\label{c:no7} \arabic{counter}.}  
Neither $G_1$ nor $G_2$ has exactly eight vertices.
 
\begin{proof}
Suppose not, say $|G_1| = 8$.  By Claim~\ref{c:nodis7},    $|S|\le 4$ and some vertex, say $x$,  in $V(G_1)\less S$ is a $7$-vertex in $G$. By Claim~\ref{c:triangles},  $G_1$   contains    $K_8^\equiv$ as a spanning subgraph.    
We next prove that $|G_2|\ge 9$. Suppose $  |G_2|\le 8$. Note that $|G_2|\ge7$ by Claim~\ref{c:d5}. If   $|G_2|=7$, then every vertex in $G_2\less S$ is a $6$-vertex. By Claims~\ref{c:triangles}, $G_2=K_7$, but then  $V(G_2)\less S$ is a clique of order $7-|S|\ge3$ because $|S|\le 4$, contrary to Claim~\ref{c:55}.    
Thus  $|G_2|=8$ and $n=16-|S|$. By Claim~\ref{c:nodis7},   some vertex, say $y$,  in $V(G_2)\less S$ is a $7$-vertex in $G$.  By Claim~\ref{c:triangles},  $G_2$ contains     $K_8^\equiv$ as a spanning subgraph.    
 Suppose $  |S|=4$. Then $G[S]=K_4$, else $G\se \kns$ by contracting $G_2\less (S\cup\{y\})$ onto an end of a missing edge of $G[S]$.   Thus $G_1=G_2=K_8^\equiv$, else say $G_1$ contains $K_8^=$ as a spanning subgraph, then $G[V(G_1)\cup\{y\}]\se\kns$. But then $n=16-4=12$ and $e(G)=e(G_1)+e(G_2)-6=50-6=44<5\times 12-14$, a contradiction.  Suppose next $ |S|= 3$.  Then $n=13$;  moreover, $G_1, G_2\in\{K_8^=, K_8^\equiv\}$, else say  $G_1 $ contains $K_8^-$ as a subgraph, then  $G[V(G_1)\cup\{y\}]\se \kns$. But then $e(G[S])\ge2$ and $e(G)=e(G_1)+e(G_2)-e(G[S]) \le 26+26-2 =50<5\times 13-14$, a contradiction. Thus $|S|\le 2$. Then \[ 5(16-|S|)-14 = e(G)=e(G_1)+e(G_2)-e(G[S])\le 28+28-e(G[S]).\] 
 It follows that $|S|=2$, $G_1=G_2=K_8$ and $G[S]=\overline{K}_2$, which is impossible.
This proves that $|G_2|\ge 9$. 
 \medskip

 Recall that  $G_1$ contains $K_8^\equiv$ as a subgraph.  It is easy to check that     $e(G[S])+d_1={{|S|}\choose2}$.    Note that if $|S|=4$, then $G_1=K_8^\equiv$, else $G\se \kns$ by contracting a component of $G_2\less S$ to a single vertex.     But then \begin{align*}
e(G_2)+d_1&=e(G)-e(G_1)+e(G[S])+d_1\\
&\ge 5n-14 -\big(28-\max\{0, 3(|S|-3)\}\big)+ {{|S|}\choose2} \\
&=  \big(5\times (n-(8-|S|))-14\big) +   5\times (8-|S|)  -\big(28-\max\{0, 3(|S|-3)\}\big)+{{|S|}\choose2}\\
&=(5 |G_2|-14) +   5\times (8-|S|)  -\big(28-\max\{0, 3(|S|-3)\}\big)+{{|S|}\choose2}\\
&\ge  5|G_2|-14, 
\end{align*}
 contrary to ($ \bigstar$) because $  |S|\le 4$ and $|G_2|\ge 9$.  
\end{proof}

Observe that, if  $|G_1|\ge 9$ and $|G_2|\ge 9$, then by ($ \bigstar$), we have 
\begin{align*}
5n - 14= e(G) &= e(G_1) + e(G_2) - e(G[S])  \\
& \le (5|G_1|-15-d_2)+(5|G_2|-15-d_1)- e(G[S]) \\
&= 5(n + |S|) - 30 - d_1 - d_2 - e(G[S]).
 \end{align*}
It follows that 
\[ 5|S| \ge 16 +  d_1 + d_2 + e(G[S])  \, \, \text{ if } |G_1|\ge 9\,  \text{ and } |G_2|\ge 9. \tag{$\blacklozenge$} \] \medskip

\noindent {\bf Claim\refstepcounter{counter}\label{c:not77} \arabic{counter}.}  
If $|G_i|=7$, then     $|G_{3-i}|\ge9$ for each $i\in[2]$. Moreover, $G[S]=K_5$ or  $G[S]=K_6$.

\begin{proof}
Suppose  $|G_1|=7$ but $|G_2|\le8$. By Claim~\ref{c:no7}, $|G_2|=7$.  
By Claim~\ref{c:triangles}, $G_1=G_2=K_7$, and every  vertex in  $V(G_1)\less S$ is a $6$-vertex in $G$. By Claim~\ref{c:55}, $|V(G_1)\less S|\le 2$ and so $|S|\ge 5$. But then $n=14-|S|\le 9$,  contrary to Claim~\ref{c:n12}.   Since $G_1= K_7$ and $1\le |V(G_1)\less S|\le 2$, we see that $G[S]=K_5$ or  $G[S]=K_6$.  
\end{proof}

\noindent {\bf Claim\refstepcounter{counter}\label{c:5conn}  \arabic{counter}.}  
$G$ is $5$-connected.
 
\begin{proof} Suppose $G$ is not 5-connected. Let $S$ be a minimal separating set of $G$, and  $G_1, G_2,   d_1, d_2$ be defined as  prior to ($ \bigstar$).  By Claim~\ref{c:not77}, 
  $|G_1|\ne 7$ and  $|G_2|\ne 7$. By Claim~\ref{c:no7}, $|G_1|\ge 9$ and $|G_2|\ge 9$. By ($\blacklozenge$), 
  $|S| \geq 4$, and so $G$ is $4$-connected. By Lemma~\ref{l:k4-}, $e(G[S])+d_1\ge5$. Note that $d_2\ge 1$ when $S$ is not a $4$-clique, and $e(G[S])=6$ when  $S$ is  a $4$-clique. In either case, we have $d_1 + d_2 + e(G[S])\ge6$, contrary to ($\blacklozenge$). \end{proof}

\noindent {\bf Claim\refstepcounter{counter}\label{c:almostclique}   \arabic{counter}.}  
  If there exists $x \in S$ such that $S \setminus x$ is a clique, then $G[S] = K_5$ or $G[S] = K_6$.
 
\begin{proof} Suppose  $S \less x$ is  a clique but $G[S]\ne K_5$ and $G[S]\ne K_6$.   Let $G_1$ and $G_2$ be as above. By Claim~\ref{c:not77}, 
  $|G_1|\ne 7$ and  $|G_2|\ne 7$. By Claim~\ref{c:no7}, $|G_1|\ge 9$ and $|G_2|\ge 9$.   By Claim~\ref{c:5conn}, $|S| \geq 5$. 
If $S$ contains a $7$-clique, then $G\se K_9^-$  by contracting a component of $G_1 \setminus S$  and  a component of $G_2 \setminus S$  to two    distinct  vertices, a contradiction. Thus    $5\le |S|\le 7$  and $S$ is not a clique.
Then  $\delta(G[S]) =d_{G[S]}(x)\leq |S| - 2$.
Since  $S\less x$ is a clique, we see that \[ d_1 = d_2 = |S| - 1 - d_{G[S]}(x)=|S| - 1 -\delta(G[S]). \]
It follows that  \[ e(G[S]) ={{|S|-1}\choose 2}+d_{G[S]}(x)= {{|S|-1}\choose 2}+\delta(G[S]). \] This, together with ($\blacklozenge$), implies that 
	\begin{align*}
	5|S| &\ge  16+d_1 + d_2 + e(G[S]) \\
	&= 16 + 2(|S| - 1 -\delta(G[S])) + {{|S|-1}\choose 2}+\delta(G[S]) \\
	&= 16+2(|S|-1)+(|S|^2 -3|S| + 2)/2 -  \delta(G[S]) \\
	&\ge 15+ (|S|^2 +|S|)/2   - (|S|-2)\\
	&=17+(|S|^2 -|S|)/2, 
	\end{align*}
which is impossible because $5\le |S|\le 7$.
\end{proof}

\noindent {\bf Claim\refstepcounter{counter}\label{c:k72}   \arabic{counter}.}  
  $G$ is  $\ket$-free.
 
\begin{proof}
Suppose $G$ has a subgraph $H$ such that $H\in \ket$.  Since $G$ is $5$-connected, 
 we obtain a $\kns$ minor in $G$ by contracting a component of $G \setminus V(H)$  to a single vertex, a contradiction. \end{proof}

\noindent {\bf Claim\refstepcounter{counter}\label{c:dnot6}   \arabic{counter}.}  
No vertex in $G$ is a $7$-vertex.
 
\begin{proof}
Suppose to the contrary that  $G$ has a $7$-vertex, say $x$.
By Claim~\ref{c:triangles}, $G[N[x]]$  contains  $K_8^\equiv$ as a spanning subgraph, contrary to Claim~\ref{c:k72}.
\end{proof}

\noindent {\bf Claim\refstepcounter{counter}\label{c:onlyone5}   \arabic{counter}.}  
$G$ has at most  two  $6$-vertices. Moreover, if $G$ has exactly two $6$-vertices, then they must be adjacent. 
 
\begin{proof}
Suppose to the contrary that $G$ has  two  distinct $6$-vertices, say $x, y$, such that $xy\notin E(G)$.
   By Claim~\ref{c:triangles},  $N[x]$ and $N[y]$   are  $7$-cliques in $G$. Then 
  $|N(x) \cap N(y)|\le 5$, else  $G[N[x] \cup N[y]]=K_8^-$, contrary to Claim~\ref{c:k72}.  
By Claim~\ref{c:5conn} and Menger's Theorem, there exist five pairwise internally vertex-disjoint $(x, y)$-paths, say  $Q_1, \ldots, Q_5$. We choose  $Q_1, \ldots, Q_5$ so that $|Q_1|+ \cdots+|Q_5|$ is as small as possible. Then each $Q_i$ contains exactly one vertex in $N(x)$ and exactly one in $N(y)$. It follows that  $G\se\km{9}{4}$ by  contracting all the edges of   $Q_1\less\{x, y\},  \ldots,  Q_5\less\{x, y\}$,   a contradiction.  This proves that  $6$-vertices  are pairwise adjacent in $G$. By Claim~\ref{c:55},  $G$ has at most  two  $6$-vertices.\end{proof}

\noindent {\bf Claim\refstepcounter{counter}\label{c:58}   \arabic{counter}.}  
No $6$-vertex   is adjacent to an $8$-vertex or  $9$-vertex   in $G$.
 
\begin{proof}
Suppose to the contrary that there exists $xy \in E(G)$ such that $d(x) = 6$ and $d(y) \in\{8,9\}$. 
By Claim~\ref{c:triangles},  $G[N[x]]=K_7$,    $N[x]\subseteq N[y]$ and 
$\delta(G[(N(y)])\ge5$. Then $d(y)=9$, otherwise $G\se G[N[y]]\se  \km{9}{4}$, a contradiction. Let $ A:=N[y]\less N[x]$. Then $|A|=3$ because $xy\in E(G)$.  Let  $A:=\{a_1, a_2, a_3\}$.  
Then either $e(\{a_1, a_2\}, N[x])\ge 10$, or $e(\{a_1, a_2\}, N[x])\ge 8$ and $a_1a_2\in E(G)$.  But then $G[N[y]\less a_3]\se\kns$ in both cases, a   contradiction.   
\end{proof}

\noindent {\bf Claim\refstepcounter{counter}\label{c:disconnected}   \arabic{counter}.}   Let $x\in V(G)$ be an $8$-vertex or $9$-vertex in $G$, and let $M$ be  the set of vertices of $N(x)$ not adjacent to all other vertices of $N(x)$.  
Then there is no component $K$ of $G \setminus N[x]$    such that $M\subseteq N(K)$.
In particular, $G \setminus N[x]$ is  disconnected if $x$ is an $8$-vertex.

\begin{proof}
Suppose such a component $K$ exists. Then every vertex in $M$ has a neighbor in $K$ because $M\subseteq N(K)$. 
 By Lemma~\ref{l:computer}, there exists  $y \in N(x)$ such that $G[N(x)] \setminus y$ has a $\km{7}{6}$ minor.
If $y\notin M$, then $y$ is complete to $N[x]\less y$ and so $G[N[x]]\se\kns$, a contradiction.
Thus $y\in M\subseteq N(K)$. 
By contracting $K$ onto $y$, we  obtain  a $\kns$ minor in $G$, a contradiction. This proves that no  such  component $K$ exists.  Suppose  $x$ is an $8$-vertex. By Claims~\ref{c:dnot6}  and \ref{c:58}, every vertex in $M$ has degree at least eight, and thus  every vertex in $M$  has a neighbor in  $G \setminus N[x]$. It follows that $G \setminus N[x]$ is disconnected, as desired. 
\end{proof}

 \noindent {\bf Claim\refstepcounter{counter}\label{c:dnot8}   \arabic{counter}.}  
No vertex in $G$ is an $8$-vertex.
 
\begin{proof}
Suppose to the contrary that  $G$ has an $8$-vertex, say $x$. Suppose $G[N(x)]$ has a $5$-clique, say $A$. By Claim~\ref{c:triangles}, $\delta(G[N(x)])\ge 5$.  It is straightforward to check that $e(G[N[x]])\ge 30=e(K_9)-6$, and so  $G[N[x]]\se\kns$, a contradiction. Thus  $G[N(x)]$  is $K_5$-free.  By Claim~\ref{c:disconnected},  let $C$ and $C'$ be two distinct  components of $G \setminus N[x]$. By Claim~\ref{c:5conn}, $|N(C)|\ge5$ and $|N(C')|\ge5$. By Claim~\ref{c:almostclique} and the fact $G[N(x)]$  is $K_5$-free, we see that 
 each of $G[N(C)]$ and $G[N(C')]$  has two independent missing edges. Let $yz$ and $uv$ be a  missing edge  of $G[N(C)]$ and $G[N(C')]$, respectively, such that $ y\ne u, v$. Note that $e(G[N[x]])\ge 8+20=e(K_9)-8$. But then  $G\se\kns$ by contracting $C$   onto  $y$ and $C'$ onto $u$,   a  contradiction. 
\end{proof}

%

\noindent {\bf Claim\refstepcounter{counter}\label{c:comporder1}   \arabic{counter}.}  
Let $x\in V(G)$ be a  $9$-vertex in $G$. 
Then  $G\less N[x]$ is disconnected. Moreover, $|C|\ge2$ for every  component  $C$ of $G \setminus N[x]$.  

\begin{proof} Suppose  $G\less N[x]$ is connected. Let $M$ be  the set of vertices of $N(x)$ not adjacent to all other vertices of $N(x)$.   
By Claims~\ref{c:dnot6},  \ref{c:58} and \ref{c:dnot8}, every vertex in $M$ has degree at least nine, and thus  every vertex in $M$  has a neighbor in  $K:=G \setminus N[x]$. But then $M\subseteq N(K)$, contrary to Claim~\ref{c:disconnected}. \medskip

Next suppose there  exists  a component $C$ of  $G \setminus N[x]$    such that $|C|=1$.   Let $y$ be the only vertex in $C$.  Suppose $y$ is not a $6$-vertex in $G$. Then $d(y)\ge 9$  by Claims~\ref{c:dnot6} and  \ref{c:dnot8}, and so   $N(C)=N(x)$, contrary to Claim~\ref{c:disconnected}.  Thus   $y$ is  a $6$-vertex in $G$. By Claim~\ref{c:triangles},  $G[N[y]]=K_7$. But then $G[ \{x\}\cup N[y]]=K_8^-$, contrary to Claim~\ref{c:k72}.
\end{proof}

\noindent {\bf Claim\refstepcounter{counter}\label{c:comp9}    \arabic{counter}.}  
 Let $x\in V(G)$ be a  $9$-vertex in $G$. 
Then  for every  component  $C$ of $G \setminus N[x]$, there exists      a vertex $v\in V(C)$ such that  $d_G(v)=9$.  

\begin{proof}
Suppose  there exists a component $C$  of $G \setminus N[x]$  such that $d_G(v)\ne 9$ for every $v\in V(C)$.  By Claim~\ref{c:comporder1}, $|C|\ge2$.  
Observe that if  all vertices in $V(C)$ are  $6$-vertices in $G$, then $|C|=2$ by Claim~\ref{c:55} and    $G[V(C)\cup N(C)]=K_7$ by Claim~\ref{c:triangles};  thus  $G[\{x\}\cup V(C)\cup N(C)]$ is not $\km{8}{3}$-free, contrary to Claim~\ref{c:k72}. Thus there exists a vertex  $y\in V(C)$ such that $d_G(y)\ne 6$.   Since $d_G(v)\ne 9$ for every $v\in V(C)$,  by Claims~\ref{c:dnot6} and  \ref{c:dnot8},   we have $d_G(y)\ge10$.  \medskip

Let $G_1 := G \setminus V(C)$ and $G_2 := G[V(C) \cup N(C)]$. Note that  $|G_2|\ge11$ and $N(C)$ is a minimal separating set of $G$. By  Claim~\ref{c:disconnected}, $N(C)\ne N(x)$. 
Thus $|C|\ge 11-8=3$.   Let  $d_1$   be defined as in the paragraph prior to Claim~\ref{c:nodis7}. 
Let $z \in N(C)$  such that  $d_{G[N(C)]}(z) = \delta(G[N(C)])$. Let $d:=d_{G[N(C)]}(z)$. 
By contracting   $G_1 \setminus N(C)$ onto $z$, we see that   $d_1 \geq |N(C)| - d - 1$. By ($ \bigstar$), 
  \[ e(G_2)  \le 5(|C| + |N(C)|) - 15 - (|N(C)| - d - 1)=5|C| + 4|N(C)| + d - 14. \tag{a} \]
Now let $t := e_G ( C, N(C) )$ and let $p\le 2$ be  the number of vertices in $V(C)$ that are $6$-vertices in $G$.    
Then $e(G_2) = e(C) + t + e(G[N(C)])$.     
Note that   $2e(C) \geq 10(|C| - p) + 6\times p - t=10|C|-4p-t$ and $2e(G[N(C)]) \geq d|N(C)|$. Thus   \[ 2e(G_2) =2e(C) +2 t + 2e(G[N(C)])\geq 10|C| - 4p + t + d|N(C)|. \tag{b} \]
Combining  (a) and (b) yields  \[ 10|C| + 8|N(C)| + 2d - 28 \ge 2e(G_2) \geq 10|C| - 4p + t + d|N(C)| \] and so \[ -t \ge d\big(|N(C)| - 2\big) - 8|N(C)| + 28-4p.\tag{c} \]
Note  that $\delta(G[N(x)]) \geq 5$ by Claim~\ref{c:triangles}, and $N(C)$ is a subset of $N(x)$, so \[ d = \delta(G[N(C)]) \geq 5 - (9 - |N(C)|) = |N(C)| - 4. \]
This, together with (c),  implies that 
	\begin{align*}
	-t &\ge \big(|N(C)| - 4\big)\big(|N(C)| - 2\big) - 8|N(C)| + 28-4p \\
	&= |N(C)|^2 - 14|N(C)| + 36-4p \\
	&= \left( |N(C)|- 7\right)^2 -13-4p,
	\end{align*}
so $-t \geq -13-4p$.
But then  \[ |C|(|C|-1)\ge 2e(C) \geq 10|C|-4p-t \geq 10|C| - 4p-13-4p=10|C|-13-8p,\]  
   Since $2e(C)$ is even, we have 
 \[ |C|(|C|-1)\ge 2e(C) \geq   10|C| - 12-8p.\tag{d}\] 
If $p\ge1$, let  $w\in V(C)$ be  a   $6$-vertex in $G$. Then $G[N_G[w]]=K_7$ and thus  $|N_G(w)\cap V(C)|\ge3$, else $G[\{x\}\cup N_G(w)]$ is not $\km{8}{3}$-free, contrary to Claim~\ref{c:k72}. Thus $|C|\ge4$ when $p\ge1$.
 Suppose $p\le1$.  Then (d)   implies that   $|C|\ge9$ and $e(C)>5|C|-14$; thus     $ G\se C\se\kns$ by the minimality of $G$,  contrary to the choice of $G$.  Thus $p=2$ and $|C|\ge4$. Then $(d)$ yields   $  |C|= 4$ or $|C|\ge7$; and $e(C)\ge 5|C|-14$. By the minimality of $G$, we have  
$  |C|= 4$ or $7\le |C|\le 8$. Let $w'\in V(C)$ be the other    $6$-vertex in $G$.   Suppose $|C|=4$. Then $V(C)$ is a $4$-clique in $G$ because $|N_G(w)\cap V(C)|\ge3$ as observed earlier. Let $y'$ be the vertex in $V(C)\less\{w,w',y\}$. Then $d_G(y')\ge10$. Recall that $N(C)\ne N(x)$. Thus there exist $x_1, x_2 \in N(x)\less N_G(w)$ such that $x_1$ is complete to $\{y', y\}$, and  $x_2$ is adjacent to $x_1$ and some vertex   in $N_G(w)\cap N(x)$.   But then $G[N[x]\cup V(C)]\se\kns$ by first contracting  the edge $x_1x_2$ to a single vertex, and    then  $G[N[x]\less (\{x_1, x_2\}\cup N_G(w)\cap N(x))] $ to another  single vertex, a contradiction. This proves that $7\le |C|\le 8$. Suppose $|C|=8$.  Then (d) implies that $e(C)\ge 5\times 8-14=e(K_8)-2$ and so $C\in\km{8}{2}$, contrary to Claim~\ref{c:k72}. Thus $|C|=7$. By    (d),  we see  that $C=K_7$ because $e(C)\ge 5\times 7-14=e(K_7)$. Note  that  each vertex in $V(C)\less\{w, w'\}$ is adjacent to at least four vertices in $N(x)$. It follows that there exists $x'\in N(x)$ such that $x'$ is adjacent to at least three vertices in $V(C)\less\{w, w'\}$.    But then $G[N[x]\cup V(C)]\se\kns$ by contracting      $G[N[x]]\less x' $ to a single vertex, a contradiction.  \end{proof}

To complete the proof, since $e(G)=5n-14$, we have $\delta(G)\le9$. By Claims~\ref{c:55}, \ref{c:dnot6} and \ref{c:dnot8}, let $x$ be a  $9$-vertex  in $G$.  By  Claim~\ref{c:comporder1},  $G \setminus N_G[x]$ is disconnected. Let $C$ be a component of $G \setminus N_G[x]$. We choose $x$ and $C$   so that $|C|$ is minimized. By Claim~\ref{c:comporder1}, $|C|\ge2$. 
By Claim~\ref{c:comp9}, $C$  contains a $9$-vertex, say  $y$, in $G$.  Note that $N_G(x)\ne N(C)$ by Claim~\ref{c:disconnected}. Thus $N_G(x)\less  N_G(y)\ne\emptyset$.  Let $K$ be the component of $G \setminus N_G[y]$ containing $x$.  Then $|K|\ge 2$ because $N_G(x)\less  N_G(y)\ne\emptyset$.  Note that $N_G(x)\cap N_G(y)\subseteq N(K)$, and every vertex in $N_G(x)\less  N_G(y)$ belongs to $K$. Let  $M$ be  the set of vertices of $N_G(y)$ not adjacent to all other vertices of $N_G(y)$.  
By Claim~\ref{c:disconnected},  $M\not\subseteq N(K)$.
  Let $z \in M \setminus N(K)$.  
Then  $z \notin N_G(x)$, else $z\in N(K)$  because $x\in V(K)$. It follows that  $z \in V(C)$. 
  Let $z'$ be a neighbor of $z$ in $G \setminus N_G[y]$. 
Note that $z' \in \big(N_G(x)\less N_G(y)\big) \cup V(C)$.     By Claims~\ref{c:dnot6}, \ref{c:58}  and \ref{c:dnot8}, we see that $d_G(z)\ge9$ and so $d_G(z')\ge9$.  
Suppose $z' \notin V(K)$. Then $z'\in V(C)$ because every vertex in $N_G(x)\less  N_G(y)$ belongs to $K$. Let  $C'$ be the component of $G \setminus N_G[y]$ that contains $z'$. Then  $|C'|\ge2$ by Claim~\ref{c:comporder1}, and $C'$ is a proper subset of $C$, contrary to our choice of $x$ and $C$. This proves that  $z' \in V(K)$, and so  $z \in N(K)$, contrary to the choice of $z$.   \medskip

This completes the proof of Theorem~\ref{t:exfun}.
\end{proof}  


\end{document}